\documentclass[a4paper, 11pt]{amsart}

\usepackage{lineno,hyperref}
\usepackage[utf8]{inputenc}

\usepackage{hyperref}
\hypersetup{nesting=true,debug=true,naturalnames=true}
\usepackage{graphicx,amssymb,upref}
\usepackage{graphicx,setspace,mathrsfs,placeins,mathtools,amssymb}

\usepackage{amsmath,amsthm}
\usepackage{amsfonts}
 \usepackage{multirow}
\usepackage{latexsym}
\usepackage{amsbsy}
\usepackage{amssymb,mathscinet}

\numberwithin{equation}{section}

\usepackage{amsmath,amsfonts,stackrel,amsthm,enumerate,subcaption}
\usepackage{graphicx,setspace,mathrsfs,placeins,mathtools,amssymb}
\usepackage{algorithm}
\usepackage{algorithmic,colortbl}
\newcommand{\algrule}[1][.2pt]{\par\vskip.5\baselineskip\hrule height #1\par\vskip.15\baselineskip}

\makeatletter
\newcommand\fs@norules{\def\@fs@cfont{\bfseries}\let\@fs@capt\floatc@ruled
	\def\@fs@pre{}%
	\def\@fs@post{}%
	\def\@fs@mid{\kern3pt}%
	\let\@fs@iftopcapt\iftrue}
\makeatother
\floatstyle{norules}
\restylefloat{algorithm}

\usepackage{amsfonts,amssymb,amscd,amsmath,enumerate,verbatim,calc,enumerate}
\theoremstyle{plain}
\newtheorem{theorem}{Theorem}[section]

\newtheorem{lemma}[theorem]{Lemma}
\newtheorem{proposition}[theorem]{Proposition}

\newtheorem{definition}[theorem]{Definition}

\modulolinenumbers[5]

\begin{document}
\title[]{Parameter estimation and long-range dependence of the fractional binomial process}
\author[]{Meena Sanjay Babulal$^*$$^a$, Sunil Kumar Gauttam$^*$$^b$ and Aditya Maheshwari$^\#$$^c$}
\email{$^a$19pmt002@lnmiit.ac.in, $^b$sgauttam@lnmiit.ac.in, $^c$adityam@iimidr.ac.in}

\address[]{
	$^*$Department of Mathematics, The LNM Institute of Information Technology, Rupa ki Nangal, Post-Sumel, Via-Jamdoli
	Jaipur 302031,
	Rajasthan, India. }
\address[]{$^\#$Operations Management and Quantitative Techniques Area,
	Indian Institute of Management Indore, Indore 453556, Madhya Pradesh, India.}
\begin{abstract}
In 1990, Jakeman (see \cite{jakeman1990statistics}) defined the binomial process  as a special case of the  classical birth-death process, where the probability of birth is proportional to the difference between a fixed number and the number of individuals present.
Later, a  fractional generalization of the binomial process was studied by Cahoy and Polito (2012) (see \cite{cahoy2012fractional}) and called it as fractional binomial process (FBP). In this paper, we study second-order properties of the FBP and the long-range behavior of the FBP and its noise process. We also estimate the parameters of the FBP using the method of moments procedure.  Finally, we present  the  simulated sample paths and its algorithm for the FBP.

\end{abstract}


  \maketitle

	\section{Introduction}

A  linear birth and death process, introduced by Feller (see \cite{Feller_1939}), is widely used to model population dynamics (see \cite{Avan_2015,bailey,pinsky2010introduction}), queuing systems (see \cite{feller2008introduction}), and other phenomena (see  \cite{articleCodeco,articlenee,articlethorne}) in which entities enter and exit a system over time.  In population model, individuals give birth to new individual  with the rate $\lambda>0$ and individuals die with rate $\mu>0$, independent of each other.
Several researchers have studied its statistical  properties  (see \cite{Chen_2011,Davison_2020,Keiding_1974,10.1111/j.2517-6161.1949.tb00032.x,Tavare2018}), and there are numerous domains in which it finds use, including  biology (see \cite{SEHL2011132,https://doi.org/10.1111/biom.12352}), 
ecology (see \cite{Avan_2015,articleCodeco,inproceedings}), 
 and finance (see \cite{Kou_2003,articlekou,Takada_2009}).\\ 

 Jakeman (see \cite{jakeman1990statistics}) studied a   linear birth and death process in which  the birth rate is proportional to  
 $N-n$, where $N$ is a fixed large number and $n$ is present population,
 while 
the mortality rate stays linear in $n$. Moreover, it is demonstrated that an equilibrium with a binomial distribution is attained as time tends to infinity, and therefore it is called the binomial process.  The behavior of the binomial process differs  from that of the traditional linear birth-death process,  since in the binomial process the  birth rate is proportional to  
 $N-n$ makes chances of birth  equal to zero whenever population size $n$ reaches $N$. Therefore, population never crosses size $N$ in the binomial process, whereas no such restriction of upper bound on population size in  the classical linear birth and death process exists. The binomial process has found its application in several areas, such as, the  telegraph wave models (see \cite{BookChapter,Jakeman_1991}),  quantum optics  (see \cite{Diament_1992,Jakeman_1995,Davenport1958AnIT}), and \textit{etc}. \\
			

  Recently, the fractional binomial process (FBP) $\left\{\mathcal{N}^{\nu}(t)\right\}_{t\geq0}$, 
 was introduced by Cahoy and 
Polito (see \cite{cahoy2012fractional}),   with birth rate $\lambda>0$ and death rate $\mu>0$. It is obtained by  taking the fractional-order derivative in place of the integer order derivative in the governing differential equation of the 
 binomial process. They also showed that one dimensional distribution of the FBP is same as  the binomial process subordinated by inverse  $\nu$-stable subordinator, $0<\nu<1$.
 It preserves the binomial limit when time tends to infinity which makes it appealing for  application in areas such as quantum optics  (see \cite{Jakeman_1991scater})
and several other disciplines (see \cite{ivanov1999multifractality,mandelbrot1963variation,climate2006}).\\\\
 Cahoy and 
Polito (see \cite{cahoy2012fractional}) have studied  several statistical properties of the FBP such as mean, variance, extinction probability and state probability. The second order  properties of the FBP still  remains to be  investigated and one of them is  the long-range dependence (LRD) property. The LRD  property of a 
 stochastic model or a process refers to the presence 
 of long-term  persistence of autocorrelation over time. More
  specifically, it means that the autocorrelation 
 function of the process decays slowly, indicating that distant observations are 
 still not uncorrelated.
	This property is in contrast to short-range 
 dependence (SRD) property, where correlation decay quickly as 
 the lag between observations increase. 
 The LRD property has found use-cases in several sub-domains like  modeling, prediction, and risk management. One can find its applications  in various fields, including finance (see \cite{lux1996stable, mandelbrot1963variation}), climate science (see \cite{climate2006}), biomedical engineering (see \cite{ ivanov1999multifractality,  peng1995quantification}),   econometrics (see \cite{pagan1996econometrics}), \textit{ etc}.  \\\\ 
In this paper, we 
prove that the FBP has the LRD 
property. Let $\delta > 0$ be fixed,  the increments of the FBP 
are defined as:$$Z^{\delta}_{\nu}(t)=\mathcal{N}^{\nu}(t+\delta)-\mathcal{N}^{\nu}(t), \qquad t\geq0,$$ which we call the fractional binomial noise (FBN). We prove that the FBN has the SRD property. \\ \\
 Parameter estimation is a fundamental aspect of data analysis, where the goal is to determine the values of unknown parameters in a model or system based on observed data. 
The parameter estimation of the the FBP is not known in the literature and in this paper we discuss the same using the method of moments. The simulated sample paths of the FBP gives an visual idea of the evolution  of the process and in this paper we present simulated sample paths for  the FBP.\\ 
			

The subsequent sections of the paper are structured as follows. In Section \ref{preliminaries}, we have stated some preliminary results regarding the binomial process
and the FBP.  Section \ref{section 3} deals with the LRD property of the FBP and the SRD property of the FBN. Section \ref{section 4} presents different simulation algorithms to create the sample path of the FBP. Finally, in Section \ref{section 5}, we have studied the parameter estimation for the FBP.

\section{Preliminaries}\label{preliminaries}
In this section, we introduce some notations, definitions and results that will be used later. 
 A linear birth-and-death (LBD) process 
 is a
continuous-time Markov chain (CTMC),
$\left\{Y(t) : t \geq 0\right\}$, defined on the countable state
space $S = \left\{0, 1, 2,3, ...\right\}$ and, the transitions
are permitted  only to  its nearest neighbours.
The state probability $p_{n}(t)=\mathbb{P}\left\{Y(t)=n|Y(0)=M\right\}$ of LBD satisfies following Cauchy problem (see \cite{bailey})
\begin{equation}
	\left\{ \begin{array}{ll}\label{(DEofproblbd)}
		\dfrac{d}{dt}p_{n}^{}(t)=\mu(n+1)p_{n+1}^{}(t)-\mu np_{n}^{}(t)-\lambda n p_{n}^{}(t)+\lambda(n-1)p_{n-1}^{}(t),\\
		p_{n}^{}(0)  = \left\{ \begin{array}{ll}
			
			1 & \mbox{$n=M,$}\\
			0 & \mbox{$n\neq M,$}
		\end{array}\right.
	\end{array}\right.
\end{equation}
where, $M\geq 1$ is the initial population. An LBD process have applications in several areas, such as,  modelling population dynamics (see \cite{Avan_2015,bailey,pinsky2010introduction}),   queuing systems (see \cite{feller2008introduction}) and biological systems (see \cite{SEHL2011132,https://doi.org/10.1111/biom.12352,Avan_2015,articleCodeco,inproceedings}).
Jakeman (see \cite{jakeman1990statistics}) studied linear birth and death process with some modifications and obtained the binomial process. We next present some preliminary results of the binomial process that will be needed later in this paper. 

\subsection{Binomial process}
  Jakeman (see \cite{jakeman1990statistics}) introduced the classical binomial process $\left\{\mathcal{N}(t)\right\}_{t\geq0}$ with birth rate $\lambda>0$ and death rate $\mu>0$ has initial value problem for the  state probability $p_{n}(t)$  given by 
     \begin{equation}
	\left\{ \begin{array}{ll}\label{(DEofprob)}
		\dfrac{d}{dt}p_{n}^{}(t)=\mu(n+1)p_{n+1}^{}(t)-\mu np_{n}^{}(t)-\lambda(N-n)p_{n}^{}(t)\\
  \qquad\qquad+\lambda(N-n+1)p_{n-1}^{}(t),\quad \mbox{if $0\leq n\leq N$, } \\  
		p_{n}^{}(0)  = \left\{ \begin{array}{ll}
			
			1 & \mbox{$n=M,$}\\
			0 & \mbox{$n\neq M,$}
		\end{array}\right.
	\end{array}\right.
\end{equation}
where $p_{n}(t)=\mathbb{P}\left\{\mathcal{N}(t)=n|\mathcal{N}(0)=M\right\}$,  $M$ is the initial population and $N$ is the maximum attainable population.
 The state space of the binomial process is $\{0,\ldots,N\}$. The 
generating function for $\mathcal{N}(t)$ is defined as 
\begin{equation}
\mbox{Q}(u,t)=\sum_{n=0}^{N}(1-u)^{n}p_{n}(t),\nonumber  
\end{equation}
and it satisfies the following partial differential equation (pde)
\begin{equation}\label{gen fun}
\left\{ \begin{array}{ll}
			\dfrac{\partial }{\partial t}\mbox{Q}(u,t)=-\mu u \dfrac{\partial }{\partial u}\mbox{Q}(u,t)-\lambda u(1-u)\dfrac{\partial }{\partial u}\mbox{Q}(u,t)-\lambda Nu\mbox{Q}(u,t) \\\\
			\mbox{Q}(u,0)=(1-u)^{M},\qquad |1-u|\leq 1,
		\end{array}\right.
\end{equation}
where $M \geq 1$ is the initial number of individuals and $N \geq M$. The solution of the above equation \eqref{gen fun} is given by 
\begin{equation}\label{Q_n}
    Q(u,t)=[1-(1-\theta)\xi u]^{N}\left(\frac{1-[(1-\theta)\xi+\theta]u}{1-(1-\theta)\xi u}\right)^{M},
\end{equation}
where $\xi=\frac{\lambda}{\lambda+\mu}$ and  $\theta(t)=\exp[-(\mu+\lambda)t].$ 
The joint probability generating function of the binomial process is given as (see \cite{jakeman1990statistics} )
\begin{align}\label{JGQ}
  Q(u,u',t)=  \sum_{n=0}^{N}(1-u)^{n}P_{n}Q_{n}(u',t),
\end{align}
where $Q_{n}(u',t)$ is given by \eqref{Q_n} and the subscript $n$ denotes the initial population. The probability of finding $n$ individuals $P_n$ given by (see \cite{jakeman1990statistics})
\begin{equation}\label{P_n}
P_n=\left\{
\begin{array}{cc} 
\binom{N}{n}
\xi^{n}(1-\xi)^{N-n} \qquad n\leq N, \\
     
     0 \qquad\qquad \qquad \qquad  \qquad n>N.
\end{array}\right.
\end{equation}
Moreover, it is observed in \cite{jakeman1990statistics} that as time tends to infinity, the evolving population follows a binomial distribution with parameters  $N$ and $\lambda/(\lambda+\mu).$ 
\ifx 199 Consider a population with a constant death rate $\mu$, and birth rate $\lambda$. Assume that the loss of individuals from the population is proportional to the current population size, while the population growth due to births is proportional to the difference between the current population size and a fixed larger number, N. This process is described by a rate equation of the form \cite{jakeman1990statistics}:
\begin{equation}
   \dfrac{d}{dt}p_{n}(t)=\mu(n+1)p_{n+1}(t)-\mu np_{n}(t)-\lambda(N-n)p_{n}(t)+\lambda(N-n+1)p_{n-1}(t)
\end{equation}
where $p_{n}(t)$  is the probability of finding $n (\leq N)$ individuals present at time $t$ and at large time this process behaves like binomial process.  A generating function for $p_{n}(t)$ defined as follows \cite{jakeman1990statistics} \begin{equation}{Q}(u,t)=\sum_{n=0}^{N}(1-u)^{n}p_{n}(t)  \end{equation}and it satiesfies the partial differential equation\begin{equation}  \dfrac{\partial }{\partial t}\mbox{Q}(u,t)=-\mu u \dfrac{\partial }{\partial t}\mbox{Q}(u,t)-\lambda u(1-u)\dfrac{\partial }{\partial t}\mbox{Q}(u,t)-\lambda Nu\mbox{Q}(u,t).\end{equation}
with the distribution 
\begin{equation}
    p_{n}(t) =\frac{1}{n!}\left(\frac{-d}{du}\right)^{n}{Q}(u,t)\big|_{ u=1}\\
 \end{equation}   
    and its $n^{th}$ factorial moments are given by 
    \begin{equation}
         \left(\frac{-d}{du}\right)^{n} {Q}(u,t)\big|_{ u=0}
    \end{equation}
   
\fi
We now state some preliminary results of the FBP that will be needed later.

\subsection{ Fractional Binomial process}
 The FBP (see \cite{cahoy2012fractional}) $\left\{\mathcal{N}^{\nu}(t)\right\}_{t\geq0}$  is obtained  by  taking fractional-order derivative in place of the integer-order derivative in  the governing differential equation  of the binomial process given in \eqref{(DEofprob)}. The governing differential equation of the  FBP $\left\{\mathcal{N}^{\nu}(t)\right\}_{t\geq0}$ with birth rate $\lambda>0$ and death rate $\mu>0$ is given by 
\begin{equation}\label{(DEoffracprob)}
		\left\{ \begin{array}{ll}
			\dfrac{d^{\nu}}{dt^{\nu}}p_{n}^{\nu}(t)=\mu(n+1)p_{n+1}^{\nu}(t)-\mu np_{n}^{\nu}(t)-\lambda(N-n)p_{n}^{\nu}(t)+\lambda(N-n+1)p_{n-1}^{\nu}(t),& \mbox{$0\leq n\leq N$,} \\  
			p_{n}^{\nu}(0)  = \left\{ \begin{array}{ll}
						1 & \mbox{$n=M$},\\
				0 & \mbox{$n\neq M$}.
			\end{array}\right.
			
		\end{array}\right.
	\end{equation}

The inverse $\nu$-stable subordinator is defined as the right-continuous of the $\nu$-stable subordinator $\{D_{\nu}(t)\}_{t\geq0}$ (see \cite{bingham71,MeerStrak13})
$$E_{\nu}(t)=\inf\{x>0|D_{\nu}(t)>x\}, \quad 0<\nu<1, \quad t\geq0.$$

 It is observed (see \cite{cahoy2012fractional}) that the one-dimensional distribution of the FBP $\left\{\mathcal{N}^{\nu}(t)\right\}_{t\geq0}$ can be written as time-changed  binomial process $\left\{\mathcal{N}(t)\right\}_{t\geq0}$  by an independent inverse $\nu$-stable subordinator $E_{\nu}(t)$,
  \textit{i.e.},
\begin{equation}\label{fbp-subord}
    \mathcal{N}^{\nu}(t)\overset{d}{=}\mathcal{N}(E_{\nu}(t)),
\end{equation}
where  $t\geq 0$ and $\nu \in (0,1).$ 
It is known that (see \cite{cahoy2012fractional}) the generating function of the FBP $Q^{\nu}(u,t)=\sum_{n=0}^{N}(1-u)^{n}p_{n}^{\nu}(t)$ solves the following differential  equation
\begin{equation}\label{fracGF}
		\left\{ \begin{array}{ll}
			\dfrac{\partial^{\nu} }{\partial t^{\nu}}\mbox{Q}^{\nu}(u,t)=-\mu u \dfrac{\partial }{\partial u}\mbox{Q}^{\nu}(u,t)-\lambda u(1-u)\dfrac{\partial }{\partial u}\mbox{Q}^{\nu}(u,t)-\lambda Nu\mbox{Q}^{\nu}(u,t),\\
			\\\mbox{Q}^{\nu}(u,0)=(1-u)^{M},\qquad |1-u|\leq 1,
		\end{array}\right.
	\end{equation}
where the initial number of individuals is $M \geq 1$, and $N \geq M$.
\begin{definition}
    Let $f(x)$ and $g(x)$ be two functions, then they are called asymptotically equivalent denoted by $f(x)\sim g(x)$, if 
    $$\lim_{x\rightarrow \infty}\frac{f(x)}{g(x)}=1.$$
\end{definition}
\noindent The Mittag–Leffler function can be defined as (see \cite{Erdelyi:381193})
$$E_{\alpha}(z)=\sum_{r=0}^{\infty}\dfrac{z^{r}}{\Gamma(\alpha r+1)}\quad \alpha, z \in \mathbb{C} , \quad \mathbb{R}(\alpha)>0.$$


\noindent Now, using  expansion of   $E_{\nu}(-x)= \tfrac{1}{\pi}  \sum_{n=0}^{\infty}\tfrac{a_{n}(\nu)}{x^{n+1}}$, where $0< \nu < 1$ (see \cite{berberan2005properties}),  we will show that $E_{\nu}(-x) $ asymptotically equivalent  to $\tfrac{a_{0}(\nu)}{\pi x}$, 
that is

\begin{align}\label{asym of ML}
 E_{\nu}(-x)&= \dfrac{1}{x\pi}\left[a_{0}(\nu)+\dfrac{a_{1}(\nu)}{x^{}}+\dfrac{a_{2}(\nu)}{x^{2}}+\cdot\cdots\right]\nonumber\\
	&= \dfrac{1}{\pi}\left[\dfrac{a_{0}(\nu)}{x}+\dfrac{a_{1}(\nu)}{x^{2}}+\dfrac{a_{2}(\nu)}{x^{3}}+\cdots\right]\nonumber\\
	&=\dfrac{1}{\pi}\left[\dfrac{a_{0}(\nu)}{x}+O\left(\dfrac{1}{x^2}\right)\right]\sim\dfrac{a_{0}(\nu)}{\pi x}.
	\end{align}


       The mean and variance of the FBP $\left\{\mathcal{N}^{\nu}(t)\right\}_{t\geq0}$ are
given by (see \cite{cahoy2012fractional} ) 
\begin{align}\label{fbp mean}
  \mathbb{E}[\mathcal{N}^{\nu}(t)]&=\left(M-\dfrac{N\lambda}{\lambda+\mu}\right)E_{\nu}(-(\lambda+\mu)t^{\nu})+\dfrac{N \lambda}{\lambda+\mu}
  \end{align}
  \begin{align}\label{fbp var}
 \mbox{Var}[\mathcal{N}^{\nu}(t)]
&=\left({\dfrac{\lambda^{2}N(N-1)}{(\lambda+\mu)^{2}}}-\dfrac{2\lambda M(N-1)}{\lambda+\mu} +M(M-1)\right)E_{\nu}(-2 (\lambda+\mu)t^{\nu})\nonumber\\ 
    &\qquad+\left(\dfrac{2\lambda^{2}N}{(\lambda+\mu)^{2}}-\dfrac{\lambda}{\lambda+\mu}(N+2M)+M\right)E_{\nu}(- (\lambda+\mu)t^{\nu})\nonumber\\ \qquad
    &\qquad
    -\left(M-N\dfrac{\lambda}{\lambda+\mu}\right)^{2}E_{\nu}(- (\lambda+\mu)t^{\nu})^{2}+\dfrac{N\lambda\mu}{(\lambda+\mu)^2},
 \end{align}


where $E_{\alpha}(\xi)=\sum_{r=0}^{\infty}\tfrac{\xi^{r}}{\Gamma(\alpha r+1)}$ is the Mittag–Leffler function.

 \subsection{ The long and short range dependence }
In the literature, there are various definitions of the LRD and SRD characteristics of a stochastic process. However, for the purpose of this paper, we will utilize the  following  definition (see \cite{ricciuti,DOVIDIO20142098,Maheshwari_Vellaisamy_2016}) for non stationary process.
 \begin{definition}\label{lrd defn}
  Let $d>0$ and $\left\{X(t)\right\}_{t\geq0}$ be a stochastic process and the asymptotic behaviour of its correlation function is given by $$\lim_{t\rightarrow \infty}\frac{\mbox{Corr}\left[X(s),X(t)\right]}{t^{-d}}=c(s), \qquad  0<s<t $$
    for fixed $s$ and $c(s) > 0$. 
   Then, we say that the stochastic process $\left\{X(t)\right\}_{t\geq0}$ has the LRD property if $d \in \left(0,1\right)$, otherwise it is said to have the SRD property when $d \in \left(1,2\right)$ .
 \ifx 
				Let $0 < s < t$ and $s$ be fixed. Assume a stochastic process $\left\{X(t)\right\}_{t\geq0}$
				has the correlation function $\mbox{Corr}\left[X(s),X(t)\right]$ that satisfies

			$$\lim_{t\rightarrow \infty}\frac{\mbox{Corr}\left[X(s),X(t)\right]}{t^{-d}}=c(s),$$
			for   large $t, d > 0$, $c_{1}(s) > $0 , $c_{2}(s) > 0$ and $c(s) > 0$ . We say $\left\{X(t)\right\}_{t\geq0}$ has the long-range dependence
			(LRD) property if $d \in \left(0,1\right)$ and has the short-range dependence (SRD) property if
			$d \in \left(1,2\right)$.
   \fi
			\end{definition}

   \section{Dependence structure for the FBP}\label{section 3}

   The aim of this section is to examine the LRD and SRD property of the FBP $\left\{\mathcal{N}^{\nu}(t)\right\}_{t\geq0}$ and the fractional binomial noise (FBN) $\left\{Z^{\delta}_{\nu}(t)\right\}_{t\geq 0}$ respectively. 
Now, we derive some results which are needed to prove it. First, we derive the recurrence relation  for the joint probability generating function (pgf) of the FBP.

\begin{lemma} 
The joint pgf of the FBP satisfies the following relationship
    \begin{align}\label{jointQ^nu}
Q^{\nu}(u,u')&=\sum_{n=0}^{N}(1-u)^{n}P_{n}Q^{\nu}_{n}(u',t),
\end{align}
where  $P_{n}$ is given by \eqref{P_n} 
  and $\mbox{Q}_{n}^{\nu}$ is the solution of \eqref{fracGF} with initial population $n$.
\end{lemma}

\begin{proof}

Let $P^{\nu}_{nn'}$ denote the probability of finding $n$ individuals present at time $t_{0}$ and $n'$ individuals at time 
$ t_{0}+ t$, and  is given by 
\begin{align}\label{pnn'}
    P^{\nu}_{nn'}&=\mathbb{P}\left( \mathcal{N}^{\nu}( t_{0}+ t)=n’, \mathcal{N}^{\nu}(t_{0})=n\right)\nonumber\\
    &=\mathbb{P}\left( \mathcal{N}^{\nu}(t_{0}+ t)=n’| \mathcal{N}^{\nu}(t_{0})=n\right)\mathbb{P}(\mathcal{N}^{\nu}(t_{0})=n).
\end{align}
Now, we have the generating function for the FBP as 
\begin{align}
Q^{\nu}(u,u';t)&=\sum_{n,n'=0}^{N}(1-u)^{n}(1-u')^{n'}P^{\nu}_{nn'}\nonumber\\
&=\sum_{n,n'=0}^{N}(1-u)^{n}(1-u')^{n'}\mathbb{P}\left( \mathcal{N}^{\nu}(t_{0}+ t)=n’| \mathcal{N}^{\nu}(t_{0})=n\right)\mathbb{P}(\mathcal{N}^{\nu}(t_{0})=n)\nonumber\mbox{ (using \eqref{pnn'}) }\\ 
&=\sum_{n=0}^{N}(1-u)^{n}\mathbb{P}(\mathcal{N}^{\nu}(t_{0})=n)\sum_{n'=0}^{N}(1-u')^{n'}\mathbb{P}\left( \mathcal{N}^{\nu}(t_{0}+ t)=n’| \mathcal{N}^{\nu}(t_{0})=n\right)\nonumber\\
&=\sum_{n=0}^{N}(1-u)^{n}\mathbb{P}(\mathcal{N}^{\nu}(t_{0})=n)\mbox{Q}^{\nu}_{n}(u',t), \nonumber 
\end{align}
using \eqref{JGQ}. Now, using $\lim_{t_0\rightarrow \infty}$$\mathbb{P}(X(t_{0})=n)=P_n$  and we have that
\begin{align}
Q^{\nu}(u,u';t)&=\sum_{n=0}^{N}(1-u)^{n}P_{n}Q^{\nu}_{n}(u',t).\qedhere
\end{align}
\end{proof} 
\noindent We next evaluate   $\mathbb{E}[\mathcal{N}^{\nu}(s)\mathcal{N}^{\nu}(t)]$ function for the FBP.

\begin{theorem}
Let  $0<\nu<1$ and  $\left\{\mathcal{N}^{\nu}(t)\right\}_{t\geq0}$  be the FBP, then 	\begin{align}\label{auto cov}
		\mathbb{E}[\mathcal{N}^{\nu}(s)\mathcal{N}^{\nu}(t)]=
		(N\xi)^{2}-N\xi(\xi-1)(E_{\nu}(-(\lambda+\mu)(t-s)^{\nu})).
	\end{align} 
\end{theorem}
\begin{proof} Using equation \eqref{jointQ^nu}, we get
\begin{align}
	\dfrac{\partial }{\partial u'}Q^{\nu}(u,u';t-s)&=\dfrac{\partial }{\partial u'}\left(\sum_{n=0}^{N}(1-u)^{n}P_{n}Q^{\nu}_{n}(u',t-s)\right)\nonumber\\
	&=\left(\sum_{n=0}^{N}(1-u)^{n}P_{n}\dfrac{\partial }{\partial u'}Q^{\nu}_{n}(u',t-s)\right).\nonumber
\end{align}
\mbox{ We have that }
\begin{align}\label{eq}
	\dfrac{\partial^{2} }{\partial u \partial u'}Q^{\nu}(u,u';t-s)&=\dfrac{\partial }{\partial u}\left[\sum_{n=0}^{N}(1-u)^{n}P_{n}\dfrac{\partial }{\partial u'}Q^{\nu}_{n}(u',t-s)\right]\nonumber\\
 &=\sum_{n=1}^{N}-n(1-u)^{n-1}P_{n}\dfrac{\partial }{\partial u'}Q^{\nu}_{n}(u',t-s))\nonumber\\
\dfrac{\partial^{2} }{\partial u \partial u'}Q^{\nu}(u,u')\big|_{u=0,u'=0}&=\sum_{n=1}^{N}-nP_{n}(-\mathbb{E}\mathbb{N}^{\nu}(t-s))\nonumber\\
  &=\sum_{n=1}^{N}n[(n-N\xi)E_{\nu}(-(\lambda+\mu)(t-s)^{\nu})+N\xi]P_{n}\nonumber\\
	&= \sum_{n=1}^{N}n^{2}(E_{\nu}(-(\lambda+\mu)(t-s)^{\nu}))P_{n}   
 -\sum_{n=1}^{N}nN\xi(E_{\nu}(-(\lambda+\mu)(t-s)^{\nu})-1)P_{n}.
\end{align}
Solving both parts separately in the above expression,  we have the following
\begin{align}\label{eqq1}
	\sum_{n=1}^{N}n^{2}(E_{\nu}(-(\lambda+\mu)t^{\nu}))P_{n}&=\sum_{n=1}^{N}n^{2}E_{\nu}(-(\lambda+\mu)t^{\nu})[^{N}C_{n}\xi^{n}(1-\xi)^{N-n}]\nonumber\\
	&=NE_{\nu}(-(\lambda+\mu)t^{\nu})\xi\sum_{n=1}^{N}\dfrac{(N-1)!}{(N-n)!(n-1)!}n\xi^{n-1}(1-\xi)^{N-n}\nonumber\\
	&=\left(N(N-1)E_{\nu}(-(\lambda+\mu)t^{\nu})\xi^{2}\sum_{n=2}^{N}\dfrac{(N-2)!}{(N-n)!(n-2)!}\xi^{n-2}(1-\xi)^{N-n}\right.\nonumber \\ &\qquad \left.+NE_{\nu}(-(\lambda+\mu)t^{\nu})\xi\sum_{n=1}^{N}\dfrac{(N-1)!}{(N-n)!(n-1)!}\xi^{n-1}(1-\xi)^{N-n}\right)\nonumber\\
	&=N(N-1)(E_{\nu}(-(\lambda+\mu)(t-s)^{\nu}))\xi^{2}
 + N(E_{\nu}(-(\lambda+\mu)(t-s)^{\nu})\xi.
\end{align}
\noindent Now, considering second part, we have that
\begin{align}\label{eqq2}
\sum_{n=1}^{N}nN\xi (E_{\nu}(-(\lambda+\mu)(t-s)^{\nu})-1)P_{n}&=\sum_{n=1}^{N}nN\xi (E_{\nu}(-(\lambda+\mu)(t-s)^{\nu})-1)[^{N}C_{n}\xi^{n}(1-\xi)^{N-n}]\nonumber\\
	&=N(E_{\nu}(-(\lambda+\mu)(t-s)^{\nu})-1)\sum_{n=1}^{N}\dfrac{N!}{(N-n)!n!}n\xi^{n+1}(1-\xi)^{N-n}\nonumber\\
	&=(N\xi)^{2}(E_{\nu}(-(\lambda+\mu)(t-s)^{\nu})-1)[1-\xi]^{N-1}\nonumber\\
	&=(N\xi)^{2}\left(E_{\nu}(-(\lambda+\mu)(t-s)^{\nu})-1\right).
\end{align}
\noindent Using equations \eqref{eq}, \eqref{eqq1} and \eqref{eqq2}, we get 
\begin{align*}
	\dfrac{\partial^{2} }{\partial u \partial u'}Q^{\nu}(u,u')\big|_{u=0,u'=0}&=\left[N(N-1)(E_{\nu}(-(\lambda+\mu)(t-s)^{\nu}))\xi^{2}+N(E_{\nu}(-(\lambda+\mu)(t-s)^{\nu})\xi\right.\nonumber \\ &\qquad \qquad \left.
      - (N\xi)^{2}\left(E_{\nu}(-(\lambda+\mu)(t-s)^{\nu})-1\right) \right]\\
	&=(N\xi)^{2}\left(E_{\nu}(-(\lambda+\mu)(t-s)^{\nu})\right)-N \xi^2 \left(E_{\nu}(-(\lambda+\mu)(t-s)^{\nu})\right)\\&\qquad+N \xi \left(E_{\nu}(-(\lambda+\mu)(t-s)^{\nu})\right)-(N\xi)^{2}\left(E_{\nu}(-(\lambda+\mu)(t-s)^{\nu})\right)+(N\xi)^{2}\\
	&=(N\xi)^{2}+N\xi(1-\xi)(E_{\nu}(-(\lambda+\mu)(t-s)^{\nu})).
\end{align*}
Hence, we have 
\begin{align*}
	\mathbb{E}[\mathcal{N}^{\nu}(s)\mathcal{N}^{\nu}(t)]& = \left(	\dfrac{\partial^{2} }{\partial u \partial u'}Q^{\nu}(u,u')\big|_{u=0,u'=0}\right)=(N\xi)^{2}+N\xi(1-\xi)(E_{\nu}(-(\lambda+\mu)(t-s)^{\nu})).\qquad\qedhere
\end{align*}
\end{proof}
\noindent Next, we compute  autocovariance function of the FBP.
\begin{theorem}
	The autocovariance function of the FBP $\left\{\mathcal{N}^{\nu}(t)\right\}_{t\geq0}$ is given by
	\begin{align}\label{cov of X(t)}
		\mbox{Cov}&[\mathcal{N}^{\nu}(s),\mathcal{N}^{\nu}(t)]=(N\xi(1-\xi)(E_{\nu}(-(\lambda+\mu)(t-s)^{\nu})))-\left((M^{2}-2MN\xi+N^{2}\xi^{2})\right.\nonumber \\ &\qquad \times \left.E_{\nu}(-(\lambda+\mu)s^{\nu})E_{\nu}(-(\lambda+\mu)t^{\nu})\right) -\left\{(M-N\xi)N\xi[E_{\nu}(-(\lambda+\mu)s^{\nu})+E_{\nu}(-(\lambda+\mu)t^{\nu})]\right\}.
	\end{align}
\end{theorem}
\begin{proof} Using  \eqref{fbp mean}, we obtain the following expression 
\begin{align*}
	\mathbb{E}[\mathcal{N}^{\nu}(s)]\mathbb{E}[\mathcal{N}^{\nu}(t)]&=(M^{2}-2MN\xi+N^{2}\xi^{2})E_{\nu}(-(\lambda+\mu)s^{\nu})E_{\nu}(-(\lambda+\mu)t^{\nu})\\& \qquad+(M-N\xi)N\xi[E_{\nu}(-(\lambda+\mu)s^{\nu})+E_{\nu}(-(\lambda+\mu)t^{\nu})]+N^{2}\xi^{2}.
\end{align*}
Using  equation (\ref{auto cov}), we get
\begin{align*}
\mbox{Cov}&[\mathcal{N}^{\nu}(s),\mathcal{N}^{\nu}(t)]=	\mathbb{E}[\mathcal{N}^{\nu}(s)\mathcal{N}^{\nu}(t)]-\mathbb{E}[\mathcal{N}^{\nu}(s)]\mathbb{E}[\mathcal{N}^{\nu}(t)]\nonumber\\
	&=((N\xi)^{2}+N\xi(1-\xi)(E_{\nu}(-(\lambda+\mu)(t-s)^{\nu})))\nonumber
 -\left((M^{2}-2MN\xi+N^{2}\xi^{2})E_{\nu}(-(\lambda+\mu)s^{\nu})
 \right.\nonumber \\ &\qquad \times\left.E_{\nu}(-(\lambda+\mu)t^{\nu})\right)-\left\{(M-N\xi)N\xi[E_{\nu}(-(\lambda+\mu)s^{\nu})+E_{\nu}(-(\lambda+\mu)t^{\nu})]\right\}-N^{2}\xi^{2} \nonumber\\
	&=(N\xi(1-\xi)(E_{\nu}(-(\lambda+\mu)(t-s)^{\nu})))\nonumber
 -\left((M^{2}-2MN\xi+N^{2}\xi^{2})E_{\nu}(-(\lambda+\mu)s^{\nu})E_{\nu}(-(\lambda+\mu)t^{\nu})\right)\\&\qquad 
 -\left\{(M-N\xi)N\xi(E_{\nu}(-(\lambda+\mu)s^{\nu})+E_{\nu}(-(\lambda+\mu)t^{\nu}))\right\}.\qedhere
\end{align*}
\end{proof}
\noindent We next present the asymptotic behavior of the variance and covariance function of the FBP.

\begin{theorem}\label{asym behaviour}
The variance and covariance functions of the FBP  are asymptotically equivalent to  
	  \begin{align}
      &\mbox{Var}[\mathcal{N}^{\nu}(t)]\sim\dfrac{a_{0}(\nu)}{\pi(\lambda+\mu)t^{\nu}}\left[\dfrac{\left(\xi^{2}N(N-1)-2\xi M(N-1)+M(M-1)\right)}{2}+\left(2\xi^{2}N-\xi(N+2M)+M\right)\right],\nonumber\\
     & \mbox{Cov}[\mathcal{N}^{\nu}(s),\mathcal{N}^{\nu}(t)]\sim\dfrac{a_{0}(\nu)}{\pi(\lambda+\mu)t^{\nu}}\left[N\xi(1-\xi)-\left((M-N\xi)^{2}E_{\nu}(-(\lambda+\mu)s^{\nu})\right)-\left\{(M-N\xi)N\xi\right\}\right],\nonumber
 \end{align}
as $t\to\infty$, where $0<\nu<1$, $0<s<t<\infty$ and $s$ is fixed.

\end{theorem}
\begin{proof} 
 Using  \eqref{fbp var}, we have
\begin{align}\label{asym var}
\mbox{Var}&[\mathcal{N}^{\nu}(t)]=\left(\xi^{2}N(N-1)-2\xi M(N-1)+M(M-1)\right)E_{\nu}(-2(\lambda+\mu)t^{\nu})\nonumber\\
	&\qquad+\left(2\xi^{2}N-\xi(N+2M)+M\right)E_{\nu}(-(\lambda+\mu)t^{\nu})
	-(M-N\xi)^{2}E_{\nu}(-(\lambda+\mu)t^{\nu})^{2}+N\xi\dfrac{\mu}{\mu+\lambda}\nonumber\\
	&\sim\left(\xi^{2}N(N-1)-2\xi M(N-1)+M(M-1)\right)\dfrac{a_{0}(\nu)}{(2\pi(\lambda+\mu)t^{\nu})}\nonumber\\ \qquad \qquad&\qquad+\left(2\xi^{2}N-\xi(N+2M)+M\right)\dfrac{a_{0}(\nu)}{\pi(\lambda+\mu)t^{\nu}}
	-(M-N\xi)^{2}\left(\dfrac{a_{0}(\nu)}{\pi(\lambda+\mu)t^{\nu}}\right)^{2}+N\xi\dfrac{\mu}{\mu+\lambda}\nonumber\\
	&\sim\dfrac{a_{0}(\nu)}{\pi(\lambda+\mu)t^{\nu}}\left[\dfrac{\left(\xi^{2}N(N-1)-2\xi M(N-1)+M(M-1)\right)}{2}%
 +\left(2\xi^{2}N-\xi(N+2M)+M\right)\right.\nonumber \\ &\qquad \qquad \qquad\qquad \qquad\left. -(M-N\xi)^{2}\left(\dfrac{a_{0}(\nu)}{\pi(\lambda+\mu)t^{\nu}}\right)^{}\right]\nonumber\\
	&\sim\dfrac{a_{0}(\nu)}{\pi(\lambda+\mu)t^{\nu}}\left[\dfrac{\left(\xi^{2}N(N-1)-2\xi M(N-1)+M(M-1)\right)}{2}+\left(2\xi^{2}N-\xi(N+2M)+M\right)\right].
\end{align}
Using  (\ref{cov of X(t)}) and \eqref{asym of ML}, we have
	\begin{align}\label{asym cov}
		\mbox{Cov}[\mathcal{N}^{\nu}(s),\mathcal{N}^{\nu}(t)]&=
		(N\xi(1-\xi)(E_{\nu}(-(\lambda+\mu)(t-s)^{\nu})))\nonumber
  -(M^{2}-2MN\xi+N^{2}\xi^{2})E_{\nu}(-(\lambda+\mu)s^{\nu})\nonumber\\&\qquad \times E_{\nu}(-(\lambda+\mu)t^{\nu}) -\left\{(M-N\xi)N\xi[E_{\nu}(-(\lambda+\mu)s^{\nu})+E_{\nu}(-(\lambda+\mu)t^{\nu})]\right\}\nonumber\\
		&\sim 	N\xi(1-\xi)\dfrac{a_{0}(\nu)}{\pi(\lambda+\mu)(t-s)^{\nu}}-\left((M-N\xi)^{2}E_{\nu}(-(\lambda+\mu)s^{\nu})\dfrac{a_{0}(\nu)}{\pi(\lambda+\mu)t^{\nu}}\right)\nonumber\\&\qquad-\left\{(M-N\xi)N\xi\left(E_{\nu}(-(\lambda+\mu)s^{\nu})+\dfrac{a_{0}(\nu)}{\pi(\lambda+\mu)t^{\nu}}\right)\right\}\nonumber\\
		&\sim 	N\xi(1-\xi)\dfrac{a_{0}(\nu)}{\pi(\lambda+\mu)(t-s)^{\nu}}-\left((M-N\xi)^{2}E_{\nu}(-(\lambda+\mu)s^{\nu})\dfrac{a_{0}(\nu)}{\pi(\lambda+\mu)t^{\nu}}\right)\nonumber\\&\qquad\qquad-\left\{(M-N\xi)N\xi\left(\dfrac{a_{0}(\nu)}{\pi(\lambda+\mu)t^{\nu}}\right)\right\}\nonumber\\
		&\sim\dfrac{a_{0}(\nu)}{\pi(\lambda+\mu)t^{\nu}}\left[\dfrac{N\xi(1-\xi)}{(1-s/t)^{\nu}}-\left((M-N\xi)^{2}E_{\nu}(-(\lambda+\mu)s^{\nu})\right)-\left\{(M-N\xi)N\xi\right\}\right]\nonumber\\
		&\sim\dfrac{a_{0}(\nu)}{\pi(\lambda+\mu)t^{\nu}}\left[N\xi(1-\xi)-\left((M-N\xi)^{2}E_{\nu}(-(\lambda+\mu)s^{\nu})\right)-\left\{(M-N\xi)N\xi\right\}\right].\qedhere
	\end{align}
\end{proof}
\ifx 
We have
\begin{align*}
	Var[\mathcal{N}^{\nu}(s)]&=\left(\xi^{2}N(N-1)-2\xi M(N-1)+M(M-1)\right)E_{\nu}(-2(\lambda+\mu)s^{\nu})\\
	&+\left(2\xi^{2}N-\xi(N+2M)+M\right)E_{\nu}(-(\lambda+\mu)s^{\nu})
	-(M-N\xi)^{2}E_{\nu}(-s^{\nu})^{2}+N\xi\dfrac{\mu}{\mu+\lambda}\\
	Var[\mathcal{N}^{\nu}(t)]&=\left(\xi^{2}N(N-1)-2\xi M(N-1)+M(M-1)\right)E_{\nu}(-2t^{\nu})\\
	&+\left(2\xi^{2}N-\xi(N+2M)+M\right)E_{\nu}(-t^{\nu})
	-(M-N\xi)^{2}E_{\nu}(-t^{\nu})^{2}+N\xi\dfrac{\mu}{\mu+\lambda}
\end{align*}
\fi

We now prove the main result of this section.
\begin{theorem}
	The FBP $\left\{\mathcal{N}^{\nu}(t)\right\}_{t\geq0}$ exhibits the LRD property.
\end{theorem}
\begin{proof} Let $0< s< t$ and using  \eqref{asym var} and \eqref{asym cov}, we get
\begin{align*}
	&\mbox{Corr}[\mathcal{N}^{\nu}(s),\mathcal{N}^{\nu}(t)]=\dfrac{\mbox{Cov}[\mathcal{N}^{\nu}(s),\mathcal{N}^{\nu}(t)]}{(\mbox{Var}[\mathcal{N}^{\nu}(s)]\mbox{Var}[\mathcal{N}^{\nu}(t)])^{1/2}} \\
	\hspace{-1cm}&\sim\dfrac{\left(\dfrac{a_{0}(\nu)}{\pi(\lambda+\mu)t^{\nu}}\right)^{1/2}\left[N\xi(1-\xi)-\left((M-N\xi)^{2}E_{\nu}(-(\lambda+\mu)s^{\nu})\right)-\left\{(M-N\xi)N\xi\right\}\right]}{\left\{\mbox{Var}[\mathcal{N}^{\nu}(s)]\left[\dfrac{\left(\xi^{2}N(N-1)-2\xi M(N-1)+M(M-1)\right)}{2}+\left(2\xi^{2}N-\xi(N+2M)+M\right)\right]\right\}^{1/2}}\\
	\hspace{-1cm}&\sim \dfrac{c(s)}{t^{\nu/2}},\\ \mbox{ where }&\\
	c(s)&=\dfrac{\left(\dfrac{a_{0}(\nu)}{(\pi(\lambda+\mu))^{}}\right)^{1/2}\left[N\xi(1-\xi)-\left((M-N\xi)^{2}E_{\nu}(-(\lambda+\mu)s^{\nu})\right)-\left\{(M-N\xi)N\xi\right\}\right]}{\left\{Var[\mathcal{N}^{\nu}(s)]\left[\dfrac{\left(\xi^{2}N(N-1)-2\xi M(N-1)+M(M-1)\right)}{2}+\left(2\xi^{2}N-\xi(N+2M)+M\right)\right]\right\}^{1/2}}.
\end{align*}
Since $\nu \in(0,1]$, the FBP has LRD property.
\end{proof}
\ifx 
Next, we derive an alternate expression for Laplace transform of \~{Q}$^{\nu}(u,t)$.
\begin{theorem}
	The Laplace transform \~{Q}$^{\nu}(u,s)=\int_{0}^{\infty}e^{-st}Q^{\nu}(u,t)dt$, $\nu\in(0,1]$, can be written as 
	$$\mbox{\~Q}^{\nu}(u,s)=\dfrac{(s^{\nu-1}+(\mu u+\lambda u(1-u)))(1-u)^{M}}{\left[s^{\nu}+(\mu u+\lambda u (1-u))s+\lambda N u \right]},$$
 where $|1-u|\leq 1.$
\end{theorem}
\begin{proof}
	From (\ref{fracGF}), we can write
	$$\dfrac{\partial^{\nu} }{\partial t^{\nu}}\mbox{Q}^{\nu}(u,t)=-\mu u \dfrac{\partial }{\partial u}\mbox{Q}^{\nu}(u,t)-\lambda u(1-u)\dfrac{\partial }{\partial u}\mbox{Q}^{\nu}(u,t)-\lambda Nu\mbox{Q}^{\nu}(u,t)$$
	Applying the Laplace transform \~{Q}$^{\nu}(u,t)=\int_{0}^{\infty}e^{-st}Q^{\nu}(u,t)dt$ , we leads to
	\begin{align}\label{LTofGF}
		[s^{\nu}\mbox{\~Q}^{\nu}(u,s)-s^{\nu-1}\mbox{Q}^{\nu}(u,0)]&=(-\mu u-\lambda u(1-u))s\mbox{\~Q}^{\nu}(u,s)\nonumber\\&-(-\mu u-\lambda u(1-u))\mbox{\~Q}^{\nu}(u,0)-\lambda N u\mbox{\~Q}^{\nu}(u,s )\nonumber \\
		\left[s^{\nu}+(\mu u+\lambda u (1-u))s+\lambda N u \right]\mbox{\~Q}^{\nu}(u,s)&=s^{\nu-1}(1-u)^{M}+(\mu u+\lambda u(1-u))(1-u)^{M}\nonumber\\
		\mbox{\~Q}^{\nu}(u,s)&=\dfrac{(s^{\nu-1}+(\mu u+\lambda u(1-u)))(1-u)^{M}}{\left[s^{\nu}+(\mu u+\lambda u (1-u))s+\lambda N u \right]}.\qedhere
	\end{align}
\fi
\ifx 
As the above function converges in $|1-u| \leq 1$, the zeros of the numerator and the denominator should coincide. Let us
indicate the zeros of the denomenator as
\begin{equation*}
	r_{12}=\dfrac{(\mu s+\lambda s-\lambda s)+_{-}\sqrt{(\mu s+\lambda s-\lambda s)^{2}-4(-\lambda s)(s^{\nu}+\mu s+\lambda N)}}{(-2\lambda s)}
\end{equation*}
with $|r_{1}(s)|<|r_{2}(s)|,\mathbb{R}(s)>0$
\begin{equation*}
	\left\{\begin{array}{ll}
r_{1}(s)+r_{2}(s)=\dfrac{-(\mu s+\lambda s-\lambda s)}{\lambda s} \\
r_{1}(s)*r_{2}(s)=\dfrac{(s^{\nu}+\mu s+\lambda N)}{\lambda s} 
	\end{array}
	\right.
\end{equation*}
By Rouch´e theorem  we have that the only zero in the unit circle is $r_{1}(s)$. Therefore it follows that  \eqref{LTofGF} can be rewritten as
\begin{equation*}
	\mbox{\~Q}^{\nu}(u,s)=\dfrac{s^{\nu-1}r_{1}(s)^{M}-r_{1}(s)^{M}}{\left[s^{\nu}+(\mu(1-r_{1}(s))+\lambda(1-r_{1}(s)) (r_{1}(s)))s+\lambda N (1-r_{1}(s)) \right]}.\qedhere
\end{equation*}

\end{proof}
\fi

\ifx
We now provide a numerical approach to solve the forward equation \eqref{fracGF}, for $u$ and $t$ close to zero. It is known (see \cite{Usero2008FractionalTS}) that an infinite sum can be used to rewrite any continuous differentiable function, $l(t, u)$, with regard to time and $u$.
\begin{align*}
    l(t, u)=\sum_{k=0}^{\infty}\frac{1}{\Gamma(k \nu +1)}\left[\left(\frac{\partial}{\partial t ^\nu}\right)^k l(t, z)\right]_{t=t_{0}}(t-t_{0})^{k \nu}
\end{align*}
In addition, if $l(u,t)$  is the product of two $C^{\infty}$ functions, $f (t)$ and $g(z)$ and  we take Taylor's series expansion $g(z)$ near $z_{0}$, then we can rewrite $l(t,u)$ as

\begin{align*}
    l(t, u)&=\sum_{k=0}^{\infty}\sum_{h=0}^{\infty}\frac{1}{\Gamma(k \nu +1)}\frac{1}{h!}\left[\left(\frac{\partial}{\partial t ^\nu}\right)^k f(t)\right]_{t=t_{0}}\left[\frac{\partial^{\nu}g(u)}{\partial u ^{\nu} }\right]_{u=u_{0}}(t-t_{0})^{k \nu}(u-u_{0})^{h}\nonumber \\
    &=\sum_{k=0}^{\infty}\frac{1}{\Gamma(k \nu +1)}w_{\nu}(k,h)(t-t_{0})^{k \nu}(u-u_{0})^{h},
\end{align*}
where $(w_{\nu}(k,h))_{k,h>0}=\frac{1}{\Gamma(k \nu +1)}\frac{1}{h!}\left[\frac{\partial^{k\nu+h}}{\partial t ^{k\nu} \partial u^{h}} l(t, u)\right]_{t=t_{0}, u=u_{0}}$ is called the spectrum of $l(t, u)$. Here we have used the notation : $\frac{\partial^{k\nu+h}}{\partial t ^{k\nu} \partial u^{h}}=\left(\frac{\partial}{\partial t ^\nu}\right)^k \frac{\partial^h}{\partial u^{h}}.$ Thus we approximate Laplace transform \~{Q}$^{\nu}(u,t)$ by the following expansion:
\begin{equation}
    \mbox{{Q}}^{\nu}(u,t)\approx\sum_{k=0}^{\infty}\sum_{h=0}^{\infty}w_{\nu}(k,h))u^{h}t^{\nu k}.
\end{equation}\label{approxLTofGF}
The differential weights in this sum are defined by:
\begin{equation}\label{diffwt}
    w_{\nu}(k,h)=\frac{1}{\Gamma(k \nu +1)}\frac{1}{h!}\left[\frac{\partial^{k\nu+h}}{\partial t ^{k\nu} \partial z^{h}} l(t, u)\right]_{t=0, u=0}.
\end{equation}
If \~{Q}$^{\nu}(u,t)$ is the product of one function of $t$ and one function of $u$, then Equation (\ref{approxLTofGF}) is exact. Iteratively, differential weights are calculated.\\
The next proposition provides some useful results.

\begin{proposition}
    The differential weights defined in \eqref{diffwt} satisfy following relations:
    \begin{align}\label{1}
        \frac{1}{\Gamma(k \nu +1)h!}\left[\frac{\partial^{k\nu+h}}{\partial t ^{k\nu} \partial u^{h}}\dfrac{\partial^{\nu}\mbox{Q}^{\nu}(u,t) }{\partial t^{\nu}} \right]_{t=0, u=0}=\frac{\Gamma((k+1) \nu +1)h!}{\Gamma(k \nu +1)h!} w_{\nu}(k+1,h)
        \end{align}
     \begin{align}\label{2}
       \hspace{-3cm} \frac{1}{\Gamma(k \nu +1)h!}\left[\frac{\partial^{k\nu+h}}{\partial t ^{k\nu} \partial u^{h}} u\mbox{Q}^{\nu}(u,t)\right]_{t=0, u=0}= w_{\nu}(k,h-1)
         \end{align}
         \begin{align}\label{3}
        \hspace{-3cm}\frac{1}{\Gamma(k \nu +1)h!}\left[\frac{\partial^{k\nu+h}}{\partial t ^{k\nu} \partial u^{h}}u \frac{\partial }{\partial u}\mbox{Q}^{\nu}(u,t) \right]_{t=0, u=0}= h w_{\nu}(k,h)
         \end{align}
         \begin{align}\label{4}
       \hspace{-2cm} \frac{1}{\Gamma(k \nu +1)h!}\left[\frac{\partial^{k\nu+h}}{\partial t ^{k\nu} \partial u^{h}}u^{2} \frac{\partial }{\partial u}\mbox{Q}^{\nu}(u,t) \right]_{t=0, u=0}= (h-1)w_{\nu}(k,h-1)
        \end{align}
\end{proposition}
\begin{proof}
    The Equation \eqref{1} follow directly from definition of differential weight. To see the proof of \eqref{2} and \eqref{3} refer to  Hainaut (see \cite{hainaut2020fractional}). The result of \eqref{4} follow from
 \begin{align*}
  \frac{\partial^{k\nu+h}}{\partial t ^{k\nu} \partial u^{h}}\left(u^{2} \frac{\partial \mbox{Q}^{\nu}(u,t) }{\partial u}\right)
  = (h-1)\frac{\partial^{k\nu+h-1}\mbox{Q}^{\nu}(u,t)}{\partial t ^{k\nu}\partial u^{h-1}} + 2u \frac{\partial^{k\nu+h+1}\mbox{Q}^{\nu}(u,t)}{\partial t ^{k\nu}\partial u^{h+1}}+u^{2}\frac{\partial^{k\nu+h+2}\mbox{Q}^{\nu}(u,t)}{\partial t ^{k\nu}\partial u^{h+2}}.
    \end{align*}
\end{proof}
The next proposition provides an approached expression for the Laplace’s transform $\mbox{Q}^{\nu}(u,t)$ (see \cite[Proposition 
 $9.3$]{hainaut2020fractional}).
\begin{proposition}
    The Laplace transform is approximated by the sum 
    \begin{align*}
         \mbox{{Q}}^{\nu}(u,t)=\sum_{k=0}^{K}\sum_{h=0}^{H}w_{\nu}(k,h))u^{h}t^{\nu k}.
    \end{align*}
         with differential weights satisfying the following recursive equation:
         \begin{align*}
             \frac{\Gamma((k+1) \nu +1)h!}{\Gamma(k \nu +1)h!} w_{\nu}(k+1,h)=-(\mu+\lambda)h w_{\nu}(k,h)+\left(\lambda h(h-1)w_{\nu}(k,h)\right)-\lambda N w_{\nu}(k,h-1)
         \end{align*}
          The initial conditions that are used for initializing the recursion are:
          \begin{align*}
            w_{\nu}(0,h)&=^{M}C_{h}u^{h} \\
             w_{\nu}(k,0)&=0, k>0.
          \end{align*}
\end{proposition}
\fi


\begin{definition}[Fractional Binomial Noise]
    
 Let $\delta > 0$ be fixed, and define the
increments of the fractional binomial process as the fractional binomial noise (FBN) is
$$Z^{\delta}_{\nu}(t)=\mathcal{N}^{\nu}(t+\delta)-\mathcal{N}^{\nu}(t),\qquad t\geq0.$$
\end{definition}

The noise process find applications in sonar communication (see \cite{SONAR}),  vehicular communications (see \cite{VEHICULARCOMMUNICATION}), wireless sensor networks (see \cite{WIRELESS})  and many other fields, where signals are transmitted through noise. We now explore the dependence structure of the fractional binomial noise (FBN) $\left\{Z^{\delta}_{\nu}(t)\right\}_{t\geq 0}$.
\begin{theorem}
	The FBN $\left\{Z^{\delta}_{\nu}(t)_{t\geq0}\right\}$ has the SRD property.
\end{theorem} 

\begin{proof}
Let $s, \delta \geq 0$ be fixed, and $0 \leq s+\delta \leq t$. We begin with
\begin{align}
   \mbox{Cov}[Z^{\delta}_{\nu}(s),Z^{\delta}_{\nu}(t)]&=\mbox{Cov}[\mathcal{N}^{\nu}(s+\delta)-\mathcal{N}^{\nu}(s),\mathcal{N}^{\nu}(t+\delta)-\mathcal{N}^{\nu}(t)]\nonumber\\
    &=		\mbox{Cov}[\mathcal{N}^{\nu}(s+\delta),\mathcal{N}^{\nu}(t+\delta)]+Cov[\mathcal{N}^{\nu}(s),\mathcal{N}^{\nu}(t)]-	\mbox{Cov}[\mathcal{N}^{\nu}(s+\delta),\mathcal{N}^{\nu}(t)]\nonumber\\&\qquad-	\mbox{Cov}[\mathcal{N}^{\nu}(s),\mathcal{N}^{\nu}(t+\delta)].
\end{align}
From \eqref{asym cov}, we have 
\begin{equation*}
     \mbox{Cov}[\mathcal{N}^{\nu}(s),\mathcal{N}^{\nu}(t)]\sim\dfrac{r}{t^{\nu}}\left[N\xi(1-\xi)-\left((M-N\xi)^{2}E_{\nu}(-(\lambda+\mu)s^{\nu})\right)-\left\{(M-N\xi)N\xi\right\}\right],
 \end{equation*}
 where $r=\frac{a_0(\nu)}{\pi(\mu+\lambda)}$. 
 Using above equation, we get
\begin{align*}
	\mbox{Cov}[Z^{\delta}_{\nu}(s),Z^{\delta}_{\nu}(t)]
	&\sim\dfrac{r}{(t+\delta)^{\nu}}\left[N\xi(1-\xi)-\left((M-N\xi)^{2}E_{\nu}(-(\lambda+\mu)(s+\delta)^{\nu})\right)-\left\{(M-N\xi)N\xi\right\}\right]\\
	&\qquad+\dfrac{r}{t^{\nu}}\left[N\xi(1-\xi)-\left((M-N\xi)^{2}E_{\nu}(-(\lambda+\mu)s^{\nu})\right)-\left\{(M-N\xi)N\xi\right\}\right]\\
	&\qquad-\dfrac{r}{t^{\nu}}\left[N\xi(1-\xi)-\left((M-N\xi)^{2}E_{\nu}(-(\lambda+\mu)(s+\delta)^{\nu})\right)-\left\{(M-N\xi)N\xi\right\}\right]\\
	&\qquad-\dfrac{r}{(t+\delta)^{\nu}}\left[N\xi(1-\xi)-\left((M-N\xi)^{2}E_{\nu}(-(\lambda+\mu)s^{\nu})\right)-\left\{(M-N\xi)N\xi\right\}\right]\\
	&\hspace{-3.2cm}\sim r(M-N\xi)^{2}\left[-\dfrac{E_{\nu}(-(\lambda+\mu)(s+\delta)^{\nu})}{(t+\delta))^{\nu}}-\dfrac{E_{\nu}(-(\lambda+\mu)s^{\nu})}{t^{\nu}}
 +\dfrac{E_{\nu}(-(\lambda+\mu)(s+\delta)^{\nu})}{t^{\nu}}+\dfrac{E_{\nu}(-(\lambda+\mu)s^{\nu})}{(t+\delta)^{\nu}}\right]\\
	&\sim r(M-N\xi)^{2}\left(\dfrac{1}{(t+\delta)^{\nu}}-\dfrac{1}{t^{\nu}}\right)\left(E_{\nu}(-(\lambda+\mu)s^{\nu})-E_{\nu}(-(\lambda+\mu)(s+\delta)^{\nu})\right)\\
	&\sim \dfrac{r(M-N\xi)^{2}}{t^{\nu}}\left(\dfrac{-\nu \delta}{t}\right)\left(E_{\nu}(-(\lambda+\mu)s^{\nu})-E_{\nu}(-(\lambda+\mu)(s+\delta)^{\nu})\right)\\
	&\sim r(M-N\xi)^{2}\left(\dfrac{-\nu \delta}{t^{1+\nu   }}\right)\left(E_{\nu}(-(\lambda+\mu)s^{\nu})-E_{\nu}(-(\lambda+\mu)(s+\delta)^{\nu})\right).
\end{align*}

Observe that 
\begin{align*}
		\mbox{Var}[Z^{\delta}_{\nu}(t)]&=\mbox{Var}[N^{\nu}(t+\delta)]+\mbox{Var}[N^{\nu}(t)]-2 \mbox{Cov}[N^{\nu}(t+\delta),N^{\nu}(t)]\\
		\mbox{Var}[\mathcal{N}^{\nu}(t)]
 &\sim\dfrac{r}{t^{\nu}}\left[\dfrac{\left(\xi^{2}N(N-1)-2\xi M(N-1)+M(M-1)\right)}{2}+\left(2\xi^{2}N-\xi(N+2M)+M\right)\right]\\
		\mbox{Cov}[\mathcal{N}^{\nu}(t+\delta),\mathcal{N}^{\nu}(t)]&=
		(N\xi(1-\xi)(E_{\nu}(-(\lambda+\mu)(\delta)^{\nu})))\\&\qquad-\left((M^{2}-2MN\xi+N^{2}\xi^{2})E_{\nu}(-(\lambda+\mu)(t+\delta))^{\nu})E_{\nu}(-(\lambda+\mu)t^{\nu})\right)\\&\qquad -\left\{(M-N\xi)N\xi[E_{\nu}(-(\lambda+\mu)(t+\delta))^{\nu}+E_{\nu}(-(\lambda+\mu)t^{\nu})]\right\}\\
		&\sim-\left(r^2(M^{2}-2MN\xi+N^{2}\xi^{2})\dfrac{1}{(t(t+\delta))^{\nu}}\right) -\left\{(M-N\xi)N\xi r\left[\dfrac{1}{(t+\delta)^{\nu}}+\dfrac{1}{t^{\nu}}\right]\right\}.
	\end{align*}
 Using above equations, we get
	\begin{align*}
		&\mbox{Var}[Z^{\delta}_{\nu}(t)]=\mbox{Var}[\mathcal{N}^{\nu}(t+\delta)]+\mbox{Var}[\mathcal{N}^{\nu}(t)]-2 \mbox{Cov}[\mathcal{N}^{\nu}(t+\delta),\mathcal{N}^{\nu}(t)]\\
		&\sim\dfrac{r}{t^{\nu}}\left[\dfrac{\left(\xi^{2}N(N-1)-2\xi M(N-1)+M(M-1)\right)}{2}+\left(2\xi^{2}N-\xi(N+2M)+M\right)\right]\\
		&\qquad+\dfrac{r}{(t+\delta)^{\nu}}\left[\dfrac{\left(\xi^{2}N(N-1)-2\xi M(N-1)+M(M-1)\right)}{2}+\left(2\xi^{2}N-\xi(N+2M)+M\right)\right]\\
		&\qquad+2\left((M^{2}-2MN\xi+N^{2}\xi^{2})r^2\dfrac{1}{(t(t+\delta))^{\nu}}\right)\\&\qquad 
  +2\left((M-N\xi)N\xi \left[\dfrac{r}{(t+\delta)^{\nu}}+\dfrac{r}{t^{\nu}}\right]\right)\\
		&\sim \dfrac{r}{t^{\nu}}\left(\left[\dfrac{\left(\xi^{2}N(N-1)-2\xi M(N-1)+M(M-1)\right)}{2}+\left(2\xi^{2}N-\xi(N+2M)+M\right)\right]\left(1+\frac{1}{\left(1+\frac{\delta}{t}\right)^{\nu}}\right) \right.\nonumber \\ &\qquad \qquad \qquad \qquad \qquad  \left.+2(M-N\xi)N\xi\left(1+\dfrac{1}{\left(1+\frac{\delta}{t}\right)^{\nu}}\right) \right)\\ 
		&\sim \dfrac{2r}{t^{\nu}}
		\left[\dfrac{\left(\xi^{2}N(N-1)-2\xi M(N-1)+M(M-1)\right)}{2^{}}+\left(2\xi^{2}N-\xi(N+2M)+M\right) +2(M-N\xi)N\xi\right].
	\end{align*}
 Now, we calculate correlation function 
\begin{align*}
	&\mbox{Corr}[Z^{\delta}_{\nu}(s),Z^{\delta}_{\nu}(t)]=\dfrac{Cov[Z^{\delta}_{\nu}(s),Z^{\delta}_{\nu}(t)]}{(Var[Z^{\delta}_{\nu}(s)]Var[Z^{\delta}_{\nu}(t)])^{1/2}}\\
	\sim&\frac{ r(M-N\xi)^{2}\left(\dfrac{-\nu \delta}{t^{1+\nu   }}\right)\left(E_{\nu}(-(\lambda+\mu)s^{\nu})-E_{\nu}(-(\lambda+\mu)(s+\delta)^{\nu})\right)}{\left\{\frac{2 r}{t^{\nu}}
			\left[\frac{\left(\xi^{2}N(N-1)-2\xi M(N-1)+M(M-1)\right)}{2}+\left(2\xi^{2}N-\xi(N+2M)+M\right)+2(M-N\xi)N\xi \right]Var[Z^{\delta}_{\nu}(s)]\right\}^{1/2}}\\\\
	&\sim\dfrac{ \left(\frac{-\nu \delta}{t^{1+\frac{\nu}{2}   }}\right)\sqrt{r}(M-N\xi)^{2}\left(E_{\nu}(-(\lambda+\mu)s^{\nu})-E_{\nu}(-(\lambda+\mu)(s+\delta)^{\nu})\right)}{\left\{2
			\left[\frac{\left(\xi^{2}N(N-1)-2\xi M(N-1)+M(M-1)\right)}{2}+\left(2\xi^{2}N-\xi(N+2M)+M\right)+2(M-N\xi)N\xi \right]Var[Z^{\delta}_{\nu}(s)]\right\}^{1/2}}\\
	&\sim\left(\frac{1}{t^{1+\frac{\nu}{2}   }}\right)c(s), \mbox{ where}
 \end{align*}
   $$ c(s)=  \dfrac{ \left(\frac{-\nu \delta}{t^{1+\frac{\nu}{2}   }}\right)\sqrt{r}(M-N\xi)^{2}\left(E_{\nu}(-(\lambda+\mu)s^{\nu})-E_{\nu}(-(\lambda+\mu)(s+\delta)^{\nu})\right)}{\left\{2
			\left[\frac{\left(\xi^{2}N(N-1)-2\xi M(N-1)+M(M-1)\right)}{2}+\left(2\xi^{2}N-\xi(N+2M)+M\right)+2(M-N\xi)N\xi \right]Var[Z^{\delta}_{\nu}(s)]\right\}^{1/2}}$$
Since $\nu \in (0,1]$, there the FBN has the SRD property.
\end{proof}

\section{simulation}\label{section 4}

In this section, we provide algorithm to simulate the FBP, which we will use in Section \ref{section 5} for parameter estimation of the FBP.\\ 
The sojourn time $S_{k}$ of the process $\left\{\mathcal{N}(t)\right\}_{t\geq0}$ is defined as the duration for which it remains in current state $k$.
The distribution of sojourn or inter-arrival time $S_{k}$ is given by (see \cite[Chapter VI, Section 3.2]{pinsky2010introduction})
$$\mathbb{P}\left\{S_{k}\geq t \right\} =\exp[-(\lambda (N-n)+\mu n)kt],$$
and thus the pdf of the sojourn time $S_{k}$ is given by  
$$f_{S_k}(t)=(\lambda (N-n)+\mu n) k\exp[-(\lambda (N-n)+\mu n)kt] ,\quad t\geq 0.$$
 Using \eqref{fbp-subord}, we  obtain the  sojourn time, $S_k^\nu, $
 for the FBP $\left\{\mathcal{N}^{\nu}(t)\right\}_{t\geq0}$ as given below

 $$\mathbb{P}\left\{S_{k}^{\nu}\geq t \right\}=E_{\nu}[-(\lambda (N-n)+\mu n) k t^{\nu}].$$
 This implies that the FBP changes state from  $k$ to $k+1$ or $k-1$  with probability  $\frac{\lambda (N-n)}{(\lambda (N-n)+\mu n)}$ or $\frac{\mu n}{(\lambda (N-n)+\mu n)},$ respectively. Now, it can be simulated using the following  procedure.


\algrule\vspace*{-.4cm}
\begin{algorithm}[H]
	\caption{Simulation of the fractional binomial process}\label{algo-fbp}
	\begin{algorithmic}[1]
		\vspace*{-.22cm}
		\algrule
		\renewcommand{\algorithmicrequire}{\textbf{Input:}}
		\renewcommand{\algorithmicensure}{\textbf{Output:}}
		\REQUIRE $N=500$,  $M=300$, $\mu=\mu n$, $\lambda=\lambda (N-n),$ $  $ and
		$\nu$.
		\ENSURE  $\mathcal{N}^{\nu}$,  simulated sample paths for the fractional binomial process.
		\\ \textit{Initialisation} : 
  $n$ is present population where $0\leq n \leq N$, $K$ is desired number of birth or death occurs  and $N$ is fixed large number.
		\FOR{$k=1:K $} 
		\STATE generate a negatively exponentially distributed random variable $\xi_{k}$ and a one sided $\nu$-stable random variable .
		\STATE simulate $S_{k}^{\nu}\overset{d}{=}{\xi_{k}^{1/\nu}} V_{\nu}$.
  \IF{$U\leq\dfrac{\lambda (N-n)}{\lambda (N-n)+\mu n}$}
  \STATE$\mathcal{N}^{\nu}(s_{k})=M+1,$\\
  \STATE otherwise $\mathcal{N}^{\nu}(s_{k})=M-1,$
	\ENDIF	
		\ENDFOR
		
		\RETURN $\mathcal{N}^{\nu}.$
		\algrule
	\end{algorithmic}\label{simu of CPP InG}
\end{algorithm}

\begin{figure}[h]	
\begin{subfigure}{8cm}
		\centering
		\includegraphics[scale=.35]{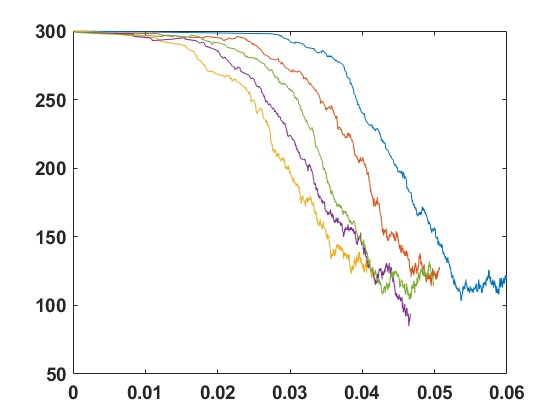}
		
		\caption{ 
  $\nu=1$,  $\lambda= 0.015$, $\mu =0.05$, $N=500$,  $M=300$}
		
	\end{subfigure}\label{ fig bp }
	\hfill
	\begin{subfigure}{8cm}
		\centering
		
		\includegraphics[scale=.35]{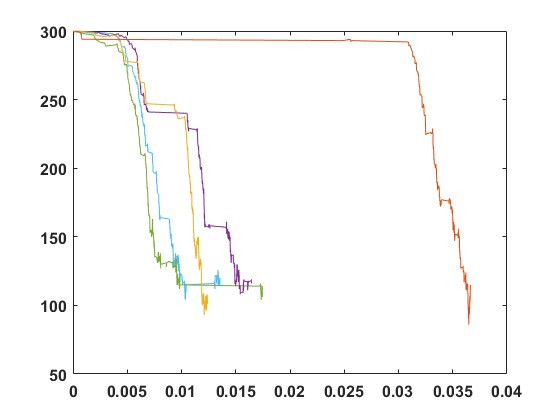}
		
		\caption{ $\nu=0.8$, $\lambda=0.015$, $\mu =0.05$, $N=500$,  $M=300$}
		
	\end{subfigure}\label{fbpb<d}\\

\caption[]{ Five simulated sample path of binomial and fractional binomial process}\label{fig 1}
\end{figure}

\begin{figure}[h]	
	\begin{subfigure}{8cm}
		\centering
		\includegraphics[scale=.35]{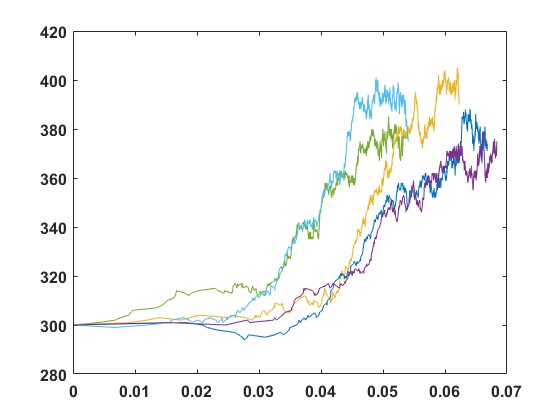}
		
		\caption{ $\nu=0.8$, $\lambda=0.05$, $\mu =0.015$, $N=500$,  $M=300$}
		
	\end{subfigure}\label{ fig fbpb>d}
	\hfill
	\begin{subfigure}{8cm}
		\centering
		
		\includegraphics[scale=.35]{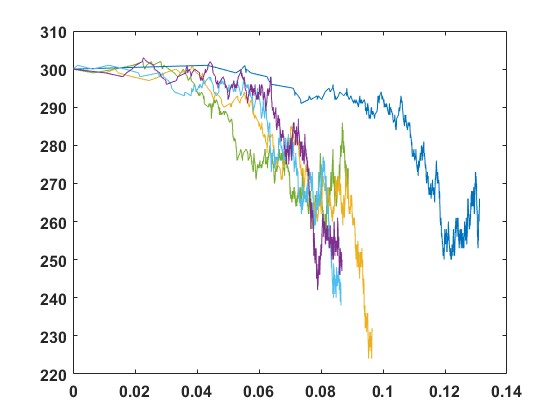}
		
		\caption{ $\nu=0.8$, $\lambda=0.015$, $\mu =0.015$, $N=500$, \\ \qquad $M=300$}
		
	\end{subfigure}\label{ fbpb=d}\\
	\caption[]{ Five simulated sample path of the fractional binomial process}\label{fig2}
\end{figure}

\ifx
Next, we give algorithm to simulate the saturable fractional linear pure birth process.
The  sojourn time $S_{j}^{\nu}$ for the process   separating $j^{th}$ and $j+1^{th}$ birth  has distribution (see \cite{cahoy2012fractional})
$$\mathbb{P}\left\{S_{j}^{\nu}\in dt \right\}=\lambda (N-j)) t^{\nu -1} E_{\nu}[-\lambda (N-j)) t^{\nu }] ,\quad t\geq 0, j\geq M,$$
  where $M$ is initial population at $t=0$.
The algorithm to  simulate  process is
 as follows. \textcolor{red}{put this reference into correct context} (see \cite{TCFPP})
 \\\\
 
\algrule\vspace*{-.4cm}
\begin{algorithm}[H]
	\caption{Simulation of the pure birth process}\label{algo-yule}
	\begin{algorithmic}[1]
		\vspace*{-.22cm}
		\algrule
		\renewcommand{\algorithmicrequire}{\textbf{Input:}}
		\renewcommand{\algorithmicensure}{\textbf{Output:}}
		\REQUIRE $C=500, j=15, M=300, \mu=1, \lambda=\mu(C-j),$ $  Y(1)=0 $ and
		$\nu$.
		\ENSURE  $Y(i)$,  simulated sample paths for the pure birth process.
		\\ \textit{Initialisation} : 
  $n=150$.
		\FOR{$i=2:n+1 $} 
		\STATE generate a uniform random variable $U_{1}, U_{2}$ and $U_{3}$ $\sim U(0, 1)$.
		\STATE set $ w(i) =\dfrac{|\log(U_{1}|)^{1/\nu}}{\lambda^{1/{\nu}}} \dfrac{\sin(\nu \pi U_{2}) \sin((1-\nu)
        \pi U_{2})]^{1/\nu-1}}{\sin(\pi U_{2})^{1/\nu} |\log(U_{3})|^{1/\nu-1}}; $.
		
		\ENDFOR
		
		\RETURN $Y.$
		\algrule
	\end{algorithmic}\label{simu of pure birth and death}
\end{algorithm}
\fi
\subsection*{Interpretation of the sample paths}
We observe from Figure \ref{fig 1} that when we compare  the binomial process with the FBP a sudden population burst (negative burst due to higher death rate) is visible in the FBP, that is, the population burst frequency increases as we decreases value of $\nu$ from $1$ to $0$. The sample paths of the FBP in Figure \ref{fig 1}- \ref{fig2}  keeps revolving around  their theoretical mean. 
\section{ Parameter estimation of the FBP}\label{section 5}

The method of moments (MoM) is a statistical technique used for estimating the parameters of the distribution of  a population. The moments are summary statistics that describe various aspects of the distribution, such as mean, variance, skewness, and kurtosis. Here, we have not used maximum likelihood estimator technique  as there is no explicit expression of  density of the FBP is available.   We can not use linear regression model for parameter estimation used by Cahoy and Polito used  in ( see \cite{MR3165549,cahoy2015transient}) as our model is non-linear regression model. 
\\\\
Let $T$ be fixed time and  $X_1, X_2,..., X_{J}$ denotes the value obtained by simulating  sample paths of the FBP. Then, using  $X_1, X_2,..., X_{J}$ we evaluate 
sample mean ($m_{1}$) and sample second moment ($m_{2})$  as follows
 \begin{align}\label{samplemoment}
    m_{1}=\dfrac{1}{J}\sum_{n=1}^{J}X_n \quad \mbox{ and } \quad m_{2}=\dfrac{1}{J}\sum_{n=1}^{J}X_n ^2.
\end{align}
 \noindent We denote the population first moment by $\mu_1'(\lambda, \nu)$ (as a function of $\lambda$ and, $\nu) $ and the population second moment by $\mu'_2(\lambda, \nu),$ then using \eqref{fbp mean} and \eqref{fbp var}, we get
\begin{align}\label{theoreticalmoments}
    &\mu_1'(\lambda, \nu)=\left(M-\dfrac{N\lambda}{\lambda+\mu}\right)E_{\nu}(-(\lambda+\mu)t^{\nu})+\dfrac{N \lambda}{\lambda+\mu} \\
     &\mu'_2(\lambda, \nu)= \left({\dfrac{\lambda^{2}N(N-1)}{(\lambda+\mu)^{2}}}-\dfrac{2\lambda M(N-1)}{\lambda+\mu} +M(M-1)\right)E_{\nu}(-2 (\lambda+\mu)t^{\nu})\nonumber\\ 
    &\qquad+\left(\dfrac{2\lambda^{2}N}{(\lambda+\mu)^{2}}-\dfrac{\lambda(N+2M)}{\lambda+\mu}+M\right)E_{\nu}(- (\lambda+\mu)t^{\nu})-\left(M-N\dfrac{\lambda}{\lambda+\mu}\right)^{2}E_{\nu}(- (\lambda+\mu)t^{\nu})^{2}\nonumber\\ \qquad
    &\qquad
    +\dfrac{N\lambda\mu}{(\lambda+\mu)^2}+\left\{\left(M-\dfrac{N\lambda}{\lambda+\mu}\right)E_{\nu}(-(\lambda+\mu)t^{\nu})+\dfrac{N \lambda}{\lambda+\mu}\right\}^2.
\end{align}

To estimate parameter $\lambda$ and $\nu$, we equate sample moments of the FBP \eqref{samplemoment} with the population moments \eqref{theoreticalmoments} and numerically solve the following equation 
\begin{align}\label{momentestimate}
   m_{1}&= \mu_1'(\lambda, \nu)\nonumber\\
   m_{2}&= \mu_2'(\lambda, \nu).
\end{align}
We took sample $J=500$ and repeated this process $K$ times, while estimating parameters, that is,
we generate this samples $X_1, X_2,..., X_{500}$ of the FBP for different sample sizes $K=100$, $1000$ and $10,000.$ 
Then, we evaluate sample mean ($m_{1,i}$) and sample second moment ($m_{2,i})$ using \eqref{momentestimate} for $i=1,2,\ldots ,K$, which gives
 $K$ estimates of $\lambda$ and $\nu$ each from above equation \eqref{momentestimate}, subsequently, we take average of $K$ estimates of $\lambda$ and $\nu $ to obtain $\hat{\lambda}$ and $\hat{\nu}$.

 Here, we have used numerical method to solve equations \eqref{momentestimate} as it is  easy to observe  from equation \eqref{theoreticalmoments} that they have complicated form and hard to solve analytically.  The tables below display these values together with associated MAD (mean absolute deviation) and MSE (mean square error). For five distinct pairs of values of $\lambda$ and $\nu$, the FBP data were simulated.

\begin{table}[]
\centering
    \caption{Parameter estimation  and its dispersion's of the FBP for parameter $\lambda = 0.3$ and $\nu = 0.8$  
with $\mu=0.5$, $M=30$ and $N=500$.}\label{tab1}
\begin{tabular}{l|lllllllll}
\multirow{2}{*}{}& \multicolumn{3}{c}{$K=100$} & \multicolumn{3}{c}{$K=1,000$} & \multicolumn{3}{c}{$K=10,000$} \\
& \multicolumn{3}{l}{\mbox{Mean} \quad \mbox{MAD}  \quad \mbox{MSE}} 
& \multicolumn{3}{l}{\mbox{Mean} \quad \mbox{MAD}  \quad {\mbox{MSE}}}  
& \multicolumn{3}{l}{\mbox{Mean} \quad \mbox{MAD}  \quad {\mbox{MSE}}}\\ \hline \\
    $\hat{\lambda}$	& 0.3045 & 0.0071 & 0.000073 & 0.3036 & 0.0046 &0.00003&0.3016 & 0.0018 & 0.00001\\
$\hat{\nu}$&0.8652 &0.1853&0.0482 &0.8588 & 0.1820& 0.0466& 0.8358 &0.0547 &0.0042
\end{tabular}
\end{table}
\begin{table}[]
\centering
   \caption{Parameter estimation  and its dispersion's of the FBP for parameter $\lambda = 0.5 $ and $\nu = 0.4$ with $\mu=0.5$, $M=30$ and $N=500$.}
    \label{tab2}
\begin{tabular}{l|lllllllll}
\multirow{2}{*}{} & \multicolumn{3}{c}{$K=100$} & \multicolumn{3}{c}{$K=1,000$} & \multicolumn{3}{c}{$K=10,000$} \\
& \multicolumn{3}{l}{\mbox{Mean} \quad \mbox{MAD}  \quad \mbox{MSE}} 
& \multicolumn{3}{l}{\mbox{Mean} \quad \mbox{MAD}  \quad {\mbox{MSE}}}  & \multicolumn{3}{l}{\mbox{Mean} \quad \mbox{MAD}  \quad {\mbox{MSE}}}\\ \hline \\ 
$ \hat{\lambda}$& 0.4942 & 0.0177 &0.00049 & 0.4962 & 0.0060 & 0.00006&0.4975 & 0.0078 &0.00009\\
$\hat{\nu}$&0.4395 &0.0754 &0.0090 &0.4206 & 0.0266 & 0.0011& 0.4175 &0.0311 & 0.0017
\end{tabular}
\end{table}

\begin{table}[]
\centering
   \caption{Parameter estimation  and its dispersion's of the FBP for 
parameter $\lambda = 0.6$ and $\nu = 0.9$ with $\mu=0.5$, $M=30$ and $N=500$.}
    \label{tab3}
\begin{tabular}{l|lllllllll}
\multirow{2}{*}{} & \multicolumn{3}{c}{$K=100$} & \multicolumn{3}{c}{$K=1,000$} & \multicolumn{3}{c}{$K=10,000$} \\
& \multicolumn{3}{l}{\mbox{Mean} \quad \mbox{MAD}  \quad \mbox{MSE}} 
& \multicolumn{3}{l}{\mbox{Mean} \quad \mbox{MAD}  \quad {\mbox{MSE}}}  & \multicolumn{3}{l}{\mbox{Mean} \quad \mbox{MAD}  \quad {\mbox{MSE}}}\\ \hline \\  
$ \hat{\lambda}$	&0.5982 & 0.0141 & 0.00005 & 0.5985 & 0.0042 & 0.00003&0.5989 & 0.0012 &0.00001\\\\
$ \hat{\nu}$&0.9062 & 0.0934 & 0.0156 &0.8930 & 0.0290 & 0.0013& 0.8938&0.0071&0.00008
\end{tabular}
\end{table}

\begin{table}[!ht]
\centering
   \caption{Parameter estimation  and its dispersion's of the FBP for 
parameter $\lambda = 0.7$ and $\nu = 0.2$ 
with  $\mu=0.5$, $M=30$ and $N=500$.}
    \label{tab4}
\begin{tabular}{l|ccccccccc}
\multirow{2}{*}{} & \multicolumn{3}{c}{$K=100$} & \multicolumn{3}{c}{$K=1,000$} & \multicolumn{3}{c}{$K=10,000$} \\
& \multicolumn{3}{l}{\mbox{Mean} \quad \mbox{MAD}  \quad \mbox{MSE}} 
& \multicolumn{3}{l}{\mbox{Mean} \quad \mbox{MAD}  \quad {\mbox{MSE}}}  & \multicolumn{3}{l}{\mbox{Mean} \quad \mbox{MAD}  \quad {\mbox{MSE}}}\\ \hline \\ 
$ \hat{\lambda}$	&  0.6845&0.0157&0.00031 &0.6874 & 0.0029 & 0.000012&0.6804&0.0009 & 0.00001
 \\ \\

$ \hat{\nu}$&0.1405 &0.0595&0.00025 & 0.1479 & 0.0521& 0.000012& 0.1501 &0.0020&0.000005

\end{tabular}
\end{table}

\begin{table}[]
\centering
   \caption{Parameter estimation  and its dispersion's of the FBP for parameter $\lambda = 0.9$ and $\nu = 0.5$  with  $\mu=0.5$, $M=30$ and $N=500$.}
    \label{tab5}
\begin{tabular}{l|lllllllll}
\multirow{2}{*}{} & \multicolumn{3}{c}{$K=100$} & \multicolumn{3}{c}{$K=1,000$} & \multicolumn{3}{c}{$K=10,000$} \\
& \multicolumn{3}{l}{\mbox{Mean} \quad \mbox{MAD}  \quad \mbox{MSE}} 
& \multicolumn{3}{l}{\mbox{Mean} \quad \mbox{MAD}  \quad {\mbox{MSE}}}  & \multicolumn{3}{l}{\mbox{Mean} \quad \mbox{MAD}  \quad {\mbox{MSE}}}\\ \hline \\ 
$\hat{\lambda}$	& 0.8877& 0.0254&0.0010 & 0.8896 & 0.0095 & 0.00014&0.8902 & 0.0028&0.00001
 \\ \\
$ \hat{\nu}$&0.5383 &0.0954&0.0095 &0.5221 & 0.0207& 0.00008& 0.5209&0.0070&0.00007\\

\end{tabular}
\end{table}

  
The estimation Tables \ref{tab1}$-$\ref{tab5} demonstrate that the relative fluctuation for  estimates of $\lambda$ and $\mu$  keep approaching true value as sample sizes increase. We also observe that the true value of parameters and estimated parameters are very close to each other and there is nearly less than $5$ percent of variation between them. It is important to keep in mind that typical sample size $K$ in many real-world applications, including network traffic data, are in the millions or more. Given the context and calculations done, we claim that our results shows robust and accurate parameter estimation. Table \ref{tab6} shows result for percent bias and coefficient of variation (CV)  based on for $1000$ simulation, where
\begin{align*}
 \mbox{Percent bias}&=   \frac{|\mbox{parameter average value- parameter value}|}{\mbox{parameter value}}*100\\
  \mbox{CV}&=\frac{\mbox{standard deviation of the estimates}}{\mbox{average estimates}}*100.
\end{align*}

\begin{table}[!ht]
 \centering
     \caption{  Percent bias and coefficient of variation for parameter $\lambda$ and $\nu.$}\label{tab6}
\begin{tabular}{l|llllll}
($\lambda, \nu$) & \multicolumn{2}{c}{$K=100$} & \multicolumn{2}{c}{$K=1,000$} & \multicolumn{2}{c}{$K=10,000$} \\
 &      Bias     &  CV        &       Bias    &     CV     &  Bias         &     CV  \\ \hline \\
$\lambda=0.3$ &  1.5146 & 2.8224
  & 1.2027& 1.8942 & 0.7977&0.7985\\
$\nu=0.8$ &  8.1469 & 25.8868 &  7.3469& 23.4176 & 4.4778& 7.8188\\ \hline\\
 $\lambda=0.5$& 1.1597& 5.0131
  & 0.7068& 4.5306 & 0.7011& 1.5336\\
$\nu=0.4$ &          9.8809 & 23.6730 &  9.4664& 21.4996 &5.1618& 8.0970\\
  \hline\\
 $\lambda=0.6$&  0.3976 & 2.8965
  & 0.3026&0.8610 &  0.2382& 0.2425\\
$\nu=0.9$ &         4.0488& 13.8720 & 0.7369& 4.1093 & 0.6852& 1.0077\\
  \hline\\
  
$\lambda=0.7$&  2.2198 & 2.7068
  & 1.8022& 0.5726 & 1.3710& 0.1674\\
$\nu=0.2$ &    29.7280& 11.8007 &  26.0687& 2.8390 & 25.7242& 1.6509\\
  \hline\\
$\lambda=0.9$&  1.3624& 3.5978
  & 1.1550& 1.3553 & 1.1537& 0.3964\\
$\nu=0.5$ &   13.2784& 18.2131  & 12.9863& 5.112  &10.2784& 1.7249\\
  \hline      
\end{tabular}
\end{table}

\subsection*{Concluding Remarks}
We have investigated that the FBP has the LRD property and its increment exhibits the SRD property. We have used the one-dimensional distributions of the FBP to simulate sample path for the process. We have derived the distribution of sojourn time of the FBP and used it to simulate sample trajectories. We have used MoM estimation technique to estimate parameters of the FBP. On comparing generated sample path  in Figure \ref{fig 1}, we can see that time taken to occur next birth or death reduces and population burst occurs. This behaviour makes the FBP more applicable in nature as such incidences occurs in real life, for example during Covid-$19$ demand of masks, sanitizer, oxygen cylinder and many other things saw a  burst in their demands. \\

\paragraph{$\mathbf{Acknowledgment}$} First author would like to acknowledge the Centre for Mathematical \& Financial Computing 
and the DST-FIST grant for the infrastructure support for the computing 
lab facility under the scheme FIST (File No: SR/FST/MS-I/2018/24) 
at the LNMIIT, Jaipur.

\bibliographystyle{plain}
\bibliography{researchbib}

 \end{document}